\documentclass[11pt]{article}
\usepackage{amssymb}
\usepackage{amsfonts}
\usepackage{amsmath}
\usepackage{mathrsfs}
\usepackage{graphicx}
\usepackage{amsbsy}
\usepackage{theorem}
\usepackage{color}
\usepackage{hyperref}
\usepackage{tikz}
\usepackage[normalem]{ulem}
\newtheorem{theorem}{Theorem}[section]

\newtheorem{proposition}[theorem]{Proposition}
\newtheorem{remark}[theorem]{Remark}

\def\thetheorem{\thesection.\arabic{theorem}}
\def\thesection{\arabic{section}}

\def\theequation {\thesection.\arabic{equation}}
\def\beq{\begin{equation}\displaystyle}
\def\eeq{\end{equation}}
\def\bel{\begin{equation} \displaystyle \begin{array}{l} }
\def\eel{\end{array} \end{equation} }
\def\bell{\begin{equation} \displaystyle \begin{array}{ll}  }
\def\eell{\end{array} \end{equation} }

\def\bea{\begin{eqnarray}}
\def\eea{\end{eqnarray} }
\def\bean{\begin{eqnarray*}}
\def\eean{\end{eqnarray*} }
\newenvironment{proof}{\noindent{\bf Proof.~}}
{{\mbox{}\hfill {\small \fbox{}}\\}}
\catcode`@=11
\renewcommand\appendix{\bigskip {\noindent \Large \bf Appendix}
  \setcounter{section}{0}%
  \setcounter{subsection}{0}%
\setcounter{equation}{0}%
\setcounter{theorem}{0}%
\def\thetheorem{A.\arabic{theorem}}
\def\theequation {A.\arabic{equation}}}
\catcode`@=12
\def\RR{\mathbb{R}}

\def\ds{\displaystyle}

\def\bs{\bigskip}

\def\bar#1{{\overline #1}}

\def\pa{\partial}

\begin{document}

\title{The sterile insect technique used as a barrier control against reinfestation}

\author{Luis Almeida
\footnote{Sorbonne Universit\'{e}, CNRS, Universit\'e de Paris, Inria, Laboratoire J.-L. Lions, F-75005 Paris, France ({\tt luis@ann.jussieu.fr}).} \and
Jorge Estrada \footnote{Laboratoire Analyse, G\'eom\'etrie et Applications CNRS UMR 7539, Universit\'e Sorbonne Paris Nord, Villetaneuse, France ({\tt estrada@math.univ-paris13.fr}).}
\and Nicolas Vauchelet\footnote{Laboratoire Analyse, G\'eom\'etrie et Applications CNRS UMR 7539, Universit\'e Sorbonne Paris Nord, Villetaneuse, France ({\tt vauchelet@math.univ-paris13.fr}).}
}

\maketitle

\begin{abstract}
  The sterile insect technique consists in massive release of sterilized males in the aim to reduce the size of mosquitoes population or even eradicate it.
  In this work, we investigate the feasability of using the sterile insect technique as a barrier against reinvasion. More precisely, we provide some numerical simulations and mathematical results showing that performing the sterile insect technique on a band large enough may stop reinvasion.
\end{abstract}

\bs

{\bf Keywords: } Sterile insect technique, wave-blocking, reaction-diffusion equations.

%{\bf 2010 AMS subject classifications: }

\bs

\section{Introduction}

Due to the number of diseases that they transmit, mosquitoes are considered as one of the most dangerous animal species for humans.
According to the World Health Organization \cite{WHO}, vector-borne diseases account for more than $17\%$ of all infectious diseases, causing more than 700 000 deaths annually. More than $3.9$ billion people in over $128$ countries are at risk of contracting dengue, with $96$ million cases estimated per year.
Malaria causes more than $400\,000$ deaths annually.

Since there is no vaccine against these diseases yet, the best strategy to control them is to act directly on the mosquito population.
Several strategies are developed and experimented to achieve this goal.
Some techniques aim at replacing the existing population of mosquitoes by a population unable to propagate the pathogens. This has triggered a growing interest in the use of the bacteria Wolbachia \cite{Hoffmann}.
Other techniques aim at reducing the size of the mosquito population like the sterile insect technique \cite{Dyck, Alphey}, the release of insects carrying a dominant lethal (RIDL) \cite{Thomas, Heinrich, Fu} and the driving of anti-pathogen genes into natural populations \cite{Gould, Marshall, Ward}.
Finally, other approaches combine both reduce and replace strategies \cite{Robert2012}.

In this article, we focus on the sterile insect technique. This strategy was introduced in the 50's by Raymond C. Bushland and Edward F. Knipling. It consists in using area-wide releases of sterile insects to reduce reproduction in a field population of the same species. Indeed wild female insects of the local population do not reproduce when they are inseminated by released sterilized males.
For mosquitoes, this technique has been successfully used to drastically reduce mosquito populations in some isolated regions (see e.g. \cite{Bossin, Zheng}).
In order to predict the dynamics of mosquito populations, mathematical modeling is an important tool. In particular, there is a growing interest in the study of control strategies (see e.g. \cite{Anguelov, Li, Bliman, MBE} and references therein).

In order to obtain rigorous results, these works usually neglect spatial dependency and only few articles propose to incorporate spatial variables in their study of the sterile insect technique.
In \cite{Lewis} the authors propose a simple scalar model to study the influence of the sterile insects density on the velocity of the spatial wave of spread of mosquitoes.
The work \cite{Oleron} focuses on the influence of the release sites and the frequency of releases in the effectiveness of the sterile insect technique and in \cite{Lee, Lee2}, the authors conduct a numerical study on some mathematical models with spatial dependency to investigate the use of a barrier zone to prevent invasion by mosquitoes.
However, to the best of our knowledge, there are no rigorous mathematical results on the existence of such barrier zones.
In this paper, we conduct a study similar to \cite{Lee} for another mathematical model which has recently been introduced in \cite{SBD}.
Moreover, we propose a strategy to rigorously prove the existence of barrier zones under appropriate conditions on the parameters.

The outline of the paper is the following: In the subsection 2.1 we introduce our dynamical system model and describe the variables and biological parameters, and also present a simplified model that results from additional assumptions. We analyze the existence of positive equilibria and the stability of the mosquito-free equilibrium. In subsection 2.2 we introduce spatial models including diffusion. In section 3 we perform numerical simulations for the spatial models (both the full and the simplified ones) to observe the existence of wave-blocking for a sufficiently large release of sterile males. In section 4 we give a sketch of a rigorous proof of the previous phenomenon, that will be presented in a forthcoming paper \cite{Jorge}, and we offer our conclusions.

\section{Mathematical model}

\subsection{Dynamical system}

Inspired by the recent paper \cite{SBD}, we propose the following mathematical model governing the dynamics of mosquitoes:
\begin{equation}\label{model1}
\begin{array}{l}
\ds \frac{dE}{dt} = b(1-\frac{E}{K}) F (1-e^{-\beta(M+\gamma M_s)}) \frac{M}{M+\gamma M_s} - (\nu_E + \mu_E) E,   \\[2mm]
\ds \frac{dM}{dt} = (1-r)\nu_E E - \mu_M M,  \\[2mm]
\ds \frac{dF}{dt} = r \nu_E E - \mu_F F, \\[2mm]
\ds \frac{dM_s}{dt}  = u - \mu_s M_s.
\end{array}
\end{equation}
In this system, the population mosquitoes is divided into several compartments. The number of mosquitoes in the aquatic phase is denoted $E$; $M$ and $F$ denote respectively the number of adult males and adult females; $M_s$ is the number of sterile male mosquitoes present, the release function being denoted by $u$.
The fraction $\frac{M}{M+ \gamma_s M_s}$ corresponds to the probability that a female mates with a wild mosquito.
Moreover, the term $(1-e^{-\beta(M+\gamma_s M_s)})$ has been introduced to model the fact that some male mosquitoes may not be fertile. It introduces a so-called Allee effect.
Finally, we have the following parameters~:
\begin{itemize}
\item $b>0$ is the oviposition rate;
\item $\mu_E>0$, $\mu_M>0$, $\mu_F>0$ and $\mu_s>0$ denote the death rates for the mosquitoes in the aquatic phase, for adults males, for adults females, and for sterile males, respectively;
\item $K$ is an environmental capacity for the aquatic phase, taking also into account the intraspecific competition;
\item $\nu_E>0$ is the rate of emergence;
\item $r\in (0,1)$ is the probability that a female emerges, then $(1-r)$ is the probability that a male emerges;
\item $u$ is a control function corresponding to the number of sterile males which are released into the field.
\end{itemize}

Since this system involves 4 equations and since we are interested in introducing the spatial dependency, it will be useful to simplify this model in order to be able to perform some rigorous mathematical analysis.
We first introduce the notations
$$
\tau = \frac{(1-r)\mu_F}{r\mu_M},
$$
and
\begin{equation}\label{eq:g}
g(F,M_s) = \frac{r\nu_E K b \tau F^2 (1-e^{-\beta(\tau F+\gamma_s M_s)})}{b\tau F^2 (1-e^{-\beta(\tau F+\gamma_s M_s)})+K(\nu_E+\mu_E)(\tau F + \gamma_s M_s)} - \mu_F F.
\end{equation}

Our first assumption concerns the male dynamics. Since males and females satisfy similar equations, it is reasonable to assume that the number of males is equal to a proportion of the number of females. Then, in order to keep the same equilibria, we assume that
\begin{equation}\label{M1}
M = \tau F.
\end{equation}
Moreover, we consider the situation in which we are in a favorable environment for mosquitos to spread. Then, we consider that the dynamics for the aquatic compartment is fast compared to the adult stage. It boils down to assume that the equation for $E$ in \eqref{model1} is at equilibrium, that is
$$
0 = b(1-\frac{E}{K}) F (1-e^{-\beta(M+\gamma M_s)}) \frac{M}{M+\gamma M_s} - (\nu_E + \mu_E) E
$$
It is equivalent to the following relation:
\begin{equation}\label{E1}
E = \frac{b  F (1-e^{-\beta(M+\gamma M_s)}) \frac{M}{M+\gamma M_s} }{ \frac{b}{K}  F (1-e^{-\beta(M+\gamma M_s)}) \frac{M}{M+\gamma M_s} + \nu_E + \mu_E}.
\end{equation}
Injecting \eqref{M1} and \eqref{E1} into the equation for $F$ in \eqref{model1}, we deduce a simplified model
\begin{equation}\label{model1_simple}
\frac{dF}{dt} = g(F,M_s), \qquad \frac{dM_s}{dt} = u - \mu_s M_s,
\end{equation}
where $g$ is defined in \eqref{eq:g}.

\begin{proposition}\label{prop:equilibria}
  Let us assume that $b r \nu_E > \mu_F (\nu_E + \mu_E)$.
  \begin{enumerate}
  \item When $u=0$.
    There exist at most two positive equilibria for systems \eqref{model1} and \eqref{model1_simple}. They are denoted $(\bar{E}_1,\bar{M}_1,\bar{F}_1,0)$ and $(\bar{E}_2,\bar{M}_2,\bar{F}_2,0)$ for \eqref{model1}, and $(\bar{F}_1,0)$ and $(\bar{F}_2,0)$ for \eqref{model1_simple}, with $0<\bar{F}_1<\bar{F}_2<\frac{Kr\nu_E}{\mu_F}$.
  \item There exists a positive constant $\tilde{U}$ large enough, such that if $u = \bar{U}$, where $\bar{U}$ is a constant such that $\bar{U}>\tilde{U}$, then the unique equilibrium for both systems \eqref{model1} and \eqref{model1_simple} is the mosquito-free equilibrium $(0, 0, 0, \bar{U}/\mu_s)$, which is globally stable.
  \end{enumerate}
\end{proposition}

\begin{proof} For $u=0$, the only equilibrium of $\frac{dM_s}{dt}  = u - \mu_s M_s$ is $M_s=0$. Substituting $M_s=0$ in \eqref{eq:g}, we get
$$g(F,0)=F\frac{b(Kr\nu_E-\mu_FF)(1-e^{-\beta\tau F})-K\mu_F(\mu_E+\nu_E)}{bF(1-e^{-\beta\tau F})+K(\mu_E+\nu_E)}$$

We prove that $g(F,0)$ has at most two positive roots. Its denominator is always positive, and its numerator can be expressed as $F((a-cF)(1-e^{-\beta\tau F})-d)$, where $a=Kbr\nu_E$, $c=b\mu_F$, $d=K\mu_F(\mu_E+\nu_E)$. Therefore, if $F>0$, then $g(F,0)=0 \iff a-d-cF=(a-cF)e^{-\beta\tau F}$. Let $g_1(F)=a-d-cF, g_2(F)=(a-cF)e^{-\beta\tau F}$.

Note that $a-d=K(br\nu_E-\mu_F(\mu_E+\nu_E))>0$ by hypothesis, and $g_1(0)=a-d, \ g_2(0)=a$. Hence, $0<g_1(0)<g_2(0)$. Moreover, we have that $g_2'(F)=e^{-\beta\tau F}(\beta\tau cF-a\beta\tau-c)$ and
$$
g_2''(F)=\beta\tau e^{-\beta\tau F}(a\beta\tau+2c-\beta\tau cF).
$$
Therefore $g_2$ has a global minimum at $F=\frac{a}{c}+\frac1{\beta\tau}$. Note that for $F=a/c$, we have $g_1(a/c)=-d<0=g_2(a/c)$, $0>g_2'(a/c)=-ce^{-a\beta/c}>g_1'(a/c)=-c$. The tangent line to $g_2(F)$ at $F=a/c$ is $g_3(F)=-ce^{-a\beta\tau/c}(F-a/c)$. Since $g_2(F)$ is convex on $(\frac{a}{c}, \frac{a}{c}+\frac1{\beta\tau})$, it follows that $g_2(F)$ lies above $g_3(F)$ on this interval. Therefore, this will also be the case for $F\geq\frac{a}{c}+\frac1{\beta\tau}$, because for  $F\geq\frac{a}{c}+\frac1{\beta\tau}$, $g_2(F)$ is increasing and $g_3(F)$ is decreasing. 

Furthermore, $g_3(F)>g_1(F)$, for all $F\geq \frac{a}{c}$. Hence, $g_1(F)<g_2(F)$, for all $F\geq \frac ac$. Therefore, any positive intersections of $g_1$ and $g_2$ must lie on $(0, a/c)$. Since $g_1$ is a straight line and $g_2$ is convex on $(0, a/c)$, they can have at most two intersections, so  $F\mapsto g(F,0)$ has at most two positive roots $0<\bar{F}_1<\bar{F}_2<\frac ac$.

Therefore, \eqref{model1_simple} has at most two equilibria $(\bar{F}_1,0)$ and $(\bar{F}_2,0)$ and substituting $\bar{F}_1, \bar{F}_2$ into \eqref{M1} and \eqref{E1} gives us the associated equilibria $(\bar{E}_1,\bar{M}_1,\bar{F}_1,0)$ and $(\bar{E}_2,\bar{M}_2,\bar{F}_2,0)$ for \eqref{model1}, which concludes the proof of the first part.\\

For the second part, for a given constant $u=\bar{U}$, the only equilibrium of $\frac{dM_s}{dt}  = u - \mu_s M_s$ is $M_s=\bar{U}/\mu_s$. On the other hand, since its denominator is always positive, the sign of $g(F, M_s)$ depends on the sign of the factor on the numerator
$$
b\tau F(Kr\nu_E-\mu_FF)(1-e^{-\beta(\tau F+\gamma_s M_s)})- K\mu_F(\nu_E+\mu_E)(\tau F + \gamma_s M_s).
$$
This numerator is negative for all $M_s\geq0, F>\frac{Kr\nu_E}{\mu_F}$. For $0\leq F<\frac{Kr\nu_E}{\mu_F}$, using the obvious inequality $0<1-e^{-\beta(\tau F+\gamma_s M_s)}\leq1$, we can bound from above this latter factor by a downward parabola in $F$,
$$
b\tau F(Kr\nu_E-\mu_FF)- K\mu_F(\nu_E+\mu_E)(\tau F + \gamma_s M_s),
$$
This parabola reaches its maximum at $F=\frac{K(br\nu_E-\mu_F(\nu_E+\mu_E))}{2b\mu_F}\in(0,\frac{Kr\nu_E}{\mu_F})$ by assumption, and the maximum value equals 
$\frac{\tau K^2(br\nu_E-\mu_F(\nu_E+\mu_E))^2}{4b\mu_F}-K\mu_F(\nu_E+\mu_E)\gamma M_s$. It is negative for all $M_s$ large enough, that is, for all
$$\bar{U}>\tilde{U}=\frac{\tau\mu_s K(br\nu_E-\mu_F(\nu_E+\mu_E))^2}{4b\mu_F^2(\nu_E+\mu_E)\gamma}.$$
(Note that if $br\nu_E\leq\mu_F(\nu_E+\mu_E)$ then this parabola, and by extension $g(F,0)$, would be negative for all $F>0$, so that $br\nu_E>\mu_F(\nu_E+\mu_E)$ is a necessary condition for $g(F,0)$ to have three non-negative equilibria.)\\

Therefore, for $\bar{U}$ large enough, so that $g(F,M_s)<0$ for all $F>0$, the only equilibrium for \eqref{model1_simple} is the mosquito-free equilibrium.
Evaluating the Jacobian at $(0, \bar{U}/\mu_s)$ we get a diagonal matrix $\binom{-\mu_F\quad 0}{0\quad -\mu_s}$ and thus the mosquito-free equilibrium is globally stable. Likewise, the Jacobian at the mosquito-free equilibrium for \eqref{model1} is a triangular matrix and its eigenvalues, that is, the diagonal  elements, are all negative:
$$
J(0,0,0,\bar{U}/\mu_s)=
  \begin{pmatrix}
    -(\nu_E+\mu_E) & 0 & 0 & 0 \\
    (1-r)\nu_E & -\mu_M & 0 & 0 \\
    r\nu_E& 0 & -\mu_F & 0 \\
    0 & 0 & 0 & -\mu_s \\
  \end{pmatrix}
$$
Therefore, the mosquito-free equilibrium is globally stable for \eqref{model1} as well.
\end{proof}

\begin{remark}
  Actually, the model proposed in \cite{SBD} is different from \eqref{model1}, since it does not consider adult females but adult females that have been fertilized $F_m$. Then the model in \cite{SBD} reads
\begin{equation}\label{model2}
\begin{array}{l}
\ds \frac{dE}{dt} = b(1-\frac{E}{K}) F_m - (\nu_E + \mu_E) E,   \\[2mm]
\ds \frac{dM}{dt} = (1-r)\nu_E E - \mu_M M,  \\[2mm]
\ds \frac{dF_m}{dt} = r \nu_E E (1-e^{-\beta(M+\gamma M_s)}) \frac{M}{M+\gamma M_s} - \mu_F F_m, \\[2mm]
\ds \frac{dM_s}{dt}  = u - \mu_s M_s.
\end{array}
\end{equation}
It is not difficult to show that the same results as the ones in Proposition \ref{prop:equilibria} holds also for this system.
It is interesting to compare the numerical results between the different models.
\end{remark}

\subsection{Spatial model}

In order to model the spatial dynamics, we consider that adult mosquitoes diffuse according to a random walk. It is classical to model this active motion by adding a diffusion operator in the adult compartments. We denote by $x$ the spatial variable. In order to simplify the approach, we only consider the one-dimensional case ($x\in\RR$). Then, all unknown functions depend now on time $t>0$ and position $x\in\RR$. The resulting model from \eqref{model1} reads
\begin{equation}\label{model_x}
\begin{array}{l}
\ds \frac{dE}{dt} = b(1-\frac{E}{K}) F(1-e^{-\beta(M+\gamma M_s)}) \frac{M}{M+\gamma M_s} - (\nu_E + \mu_E) E,   \\[2mm]
\ds \pa_t M - D_u \pa_{xx} M = (1-r)\nu_E E - \mu_M M,  \\[2mm]
\ds \pa_t F - D_u \pa_{xx} F = r \nu_E E  - \mu_F F, \\[2mm]
\ds \pa_t M_s - D_ u \pa_{xx} M_s = u - \mu_s M_s.
\end{array}
\end{equation}
In this model $D_u$ is a given diffusion coefficient (which, for simplicity, in this work is assumed to be the same for the three adult populations, but we can also consider more general cases).

Since it is hard to obtain analytical results for this system, we consider the simplified model deduced from \eqref{model1_simple}
\begin{equation}\label{model_x_simple}
\begin{array}{l}
  \ds \pa_t F - D_u \pa_{xx} F = g(F,M_s), \\[2mm]
  \ds \pa_t M_s - D_u \pa_{xx} M_s = u - \mu_s M_s.
\end{array}
\end{equation}
% observation is that when there is no sterile male injected, that is $M_s=0$, then system \eqref{model_x_simple} reduces to a the scalar reaction diffusion equation
%   $$
%   \pa_t F - D_u \pa_{xx} F = g(F,0),
%   $$
%   where the right hand side is monostable, i.e. $g(0,0)=0$, $g(\bar{F},0)$ and $g(F,0)>0$ for $F\in(0,\bar{F})$.
%   Hence, it is a classical monostable scalar reaction-diffusion equation.
%   It is well-known (see e.g. \cite{book}) that there exists a traveling wave solution for this latter scalar reaction diffusion equation with a minimal traveling speed given $c^* := 2\sqrt{D_u \frac{d}{dF} (g(F,0))_{|F=0}} = 2 \sqrt{D_u(\frac{r\nu_E b}{\nu_E+\mu_E}-\mu_F)} > 0$ where we recall that $r \nu_E b - \mu_F (\nu_E+\mu_E)$.
%   Hence, when $M_s=0$, we have a propogation wave of mosquitoes invading the whole domain.  \\

An important observation in the case when $M_s$ is a non-negative constant, is that system \eqref{model_x_simple} simplifies into a scalar reaction-diffusion equation with a bistable right hand side:
  $$
  \pa_t F - D_u \pa_{xx} F = g(F,M_s).
  $$
  Indeed, we have seen in Proposition \ref{prop:equilibria} that there exists $\bar{M_s}$ such that for any $0\leq M_s < \bar{M_s}$, the function $F\mapsto g(F,M_s)$ admits two positive roots $\bar{F_1}$ and $\bar{F_2}$ and for any $M_s\in [0,\bar{M_s})$, we have $g(F,M_s)<0$ for $F\in (0,\bar{F}_1)$, and $g(F,M_s)>0$ for $F\in (\bar{F}_1,\bar{F}_2)$.  \\
  It is now well-established (see e.g. \cite{book}) that there exists an unique traveling wave with a speed which has the same sign as the quantity $\int_0^{\bar{F}_2} g(F,M_s) \,dF$.
  Then, one possibility to avoid spreading of mosquitoes is to investigate the possibility to find a constant $M_s$ such that $\int_0^{\bar{F}_2} g(F,M_s) \,dF<0$.
  Such problem has been investigated in \cite{LewisVDDriessche}.

  In order to illustrate this observation, we perform some numerical computations of models \eqref{model_x} and \eqref{model_x_simple}. These models are discretized thanks to a finite difference scheme on an uniform grid mesh. We use the numerical values in Table \eqref{table}.
  We display in Figure \ref{fig_TW} a comparison between numerical solutions for the two models \eqref{model_x} and \eqref{model_x_simple}.

  \begin{figure}
    \includegraphics[width=5.5cm]{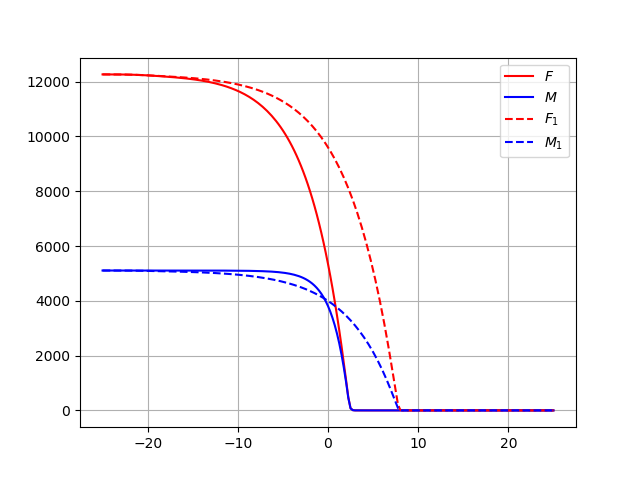}\hfill
    \includegraphics[width=5.5cm]{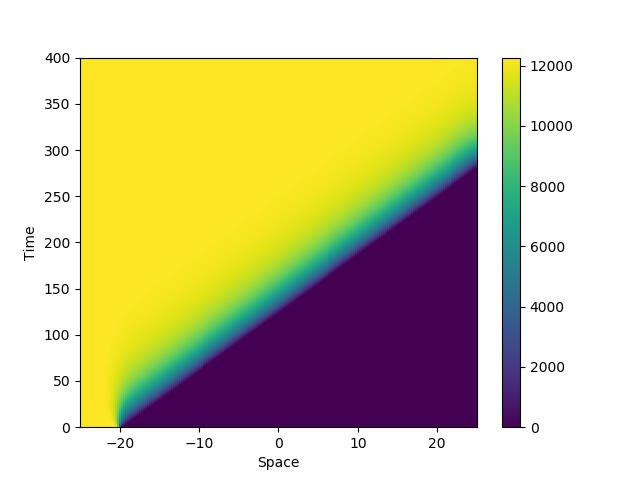}\hfill
    \includegraphics[width=5.5cm]{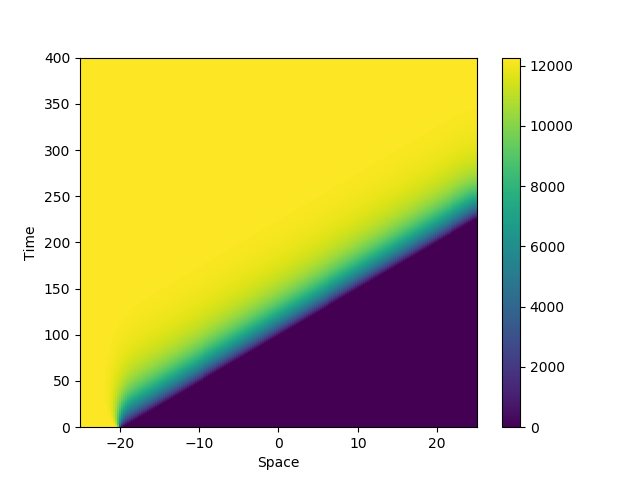}
    \caption{Comparison of the dynamics of $(F,M)$ solving \eqref{model_x} and $(F_1,M_1)$ solving the simplified model \eqref{model_x_simple}. Left : Profiles of solutions at time $T=140$ with same initial data. Center : Dynamics in time and space of the females density for system \eqref{model_x}. Right : Dynamics in time and space of the females density for system \eqref{model_x_simple}. }\label{fig_TW}
  \end{figure}

  \begin{remark}
    The spatial model for \eqref{model2} reads (in the case where the diffusion of all adult types is the same, as before)
    \begin{equation}\label{model2_x}
      \begin{array}{l}
        \ds \frac{dE}{dt} = b(1-\frac{E}{K}) F_m - (\nu_E + \mu_E) E,   \\[2mm]
        \ds \pa_t M - D_u \pa_{xx} M = (1-r)\nu_E E - \mu_M M,  \\[2mm]
        \ds \pa_t F_m - D_u \pa_{xx} F_m = r \nu_E E (1-e^{-\beta(M+\gamma M_s)}) \frac{M}{M+\gamma M_s} - \mu_F F_m, \\[2mm]
        \ds \pa_t M_s - D_u \pa_{xx} M_s  = u - \mu_s M_s.
      \end{array}
    \end{equation}
  \end{remark}

In this paper, we want to investigate the possibility of blocking the propagation  of the spreading of mosquitoes by releasing sterile mosquitoes on a band of width $L$. May the sterile insect technique be used to act as a barrier to avoid re-invasion of mosquitoes in a free-mosquito region ?
In order to answer this question, we first perform some numerical simulations in the next section.

\section{Numerical simulations}

We choose $u(t,x) = \bar{U} \mathbf{1}_{[0,L]}(x)$, where $\bar{U}$ is a given positive constant. We propose some numerical simulations.
As above we implement a finite difference scheme on an uniform grid.
The values of the numerical parameters are taken from \cite{SBD} and are given in Table \ref{table}.

\begin{table}[h]
\centering\begin{tabular}{|l|c|c|c|c|c|c|c|c|c|c|}
  \hline
  Parameter & $\beta$ & b & r & $\mu_E$ & $\nu_E$ & $\mu_F$ & $\mu_M$ & $\gamma_s$ & $\mu_s$ & $D_u$  \\
  \hline
  Value & $10^{-2}$ & $10$ & $0.49$ & $0.03$ & $0.05$ & $0.04$ & $0.1$ & $1$ & $0.12$ & $0.0125$ \\
  \hline
\end{tabular}
\caption{Table of the numerical values used for the numerical simulations. These values are taken from \cite{SBD}}
\label{table}
\end{table}

We present in Figures \ref{fig1} and \ref{fig2} the dynamics in time and space of the female density $F$ for models \eqref{model_x} and \eqref{model_x_simple}, respectively. In both figures, we assume that the domain where the release of sterile males is perform is of width $L=5\text{ km}$. The release intensity is $\bar{U}=10\,000\text{ km}^{-2}$ (left), $\bar{U}=15\,000\text{ km}^{-2}$ (center), $\bar{U}=20\,000\text{ km}^{-2}$ (right).
We first notice that it seems that for sufficiently large $\bar{U}$, the mosquito wave is not able to pass through the release zone. On the contrary, if $\bar{U}$ is not large enough, the wave is only delayed by the release zone.
Comparing Figures \ref{fig1} and \ref{fig2}, we notice that the delay is more important for the solution of model \eqref{model_x} than for the solution of the simplified model \eqref{model_x_simple}. This is not surprising since we have already observed in Figure \ref{fig_TW} that the wave propagation is faster for the simplified model.

\begin{figure}[h!]
  \includegraphics[width=5.5cm]{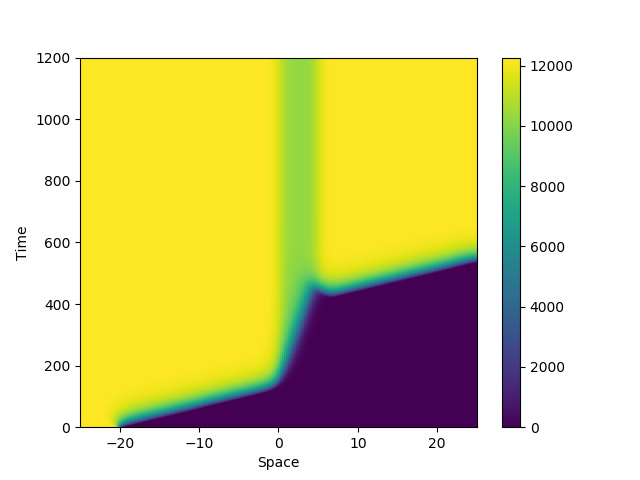}\hfill
  \includegraphics[width=5.5cm]{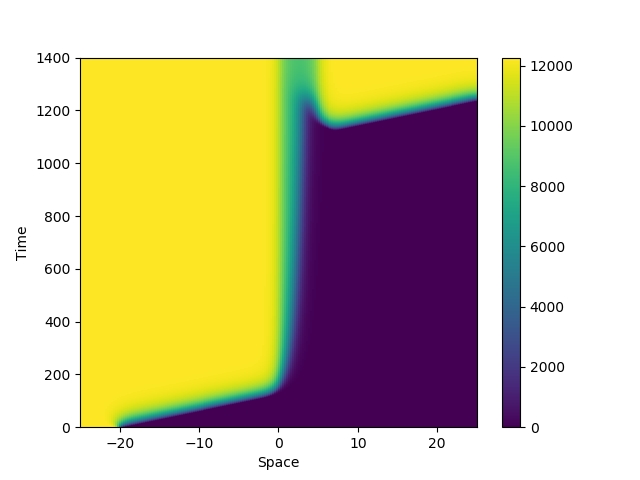}\hfill
  \includegraphics[width=5.5cm]{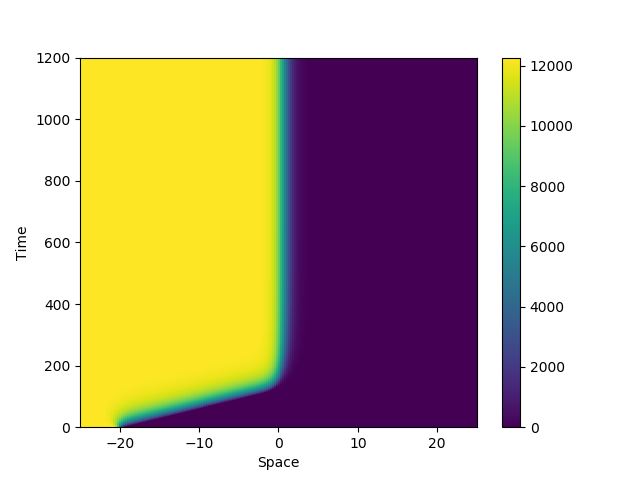}
  \caption{Numerical simulations of system \eqref{model_x} with $L=5$ km and $\bar{U}=10\,000\text{ km}^{-2}$ (left), $\bar{U}=15\,000\text{ km}^{-2}$ (center), and $\bar{U}=20\,000\text{ km}^{-2}$ (right).}\label{fig1}
\end{figure}

\begin{figure}[h!]
  \includegraphics[width=5.5cm]{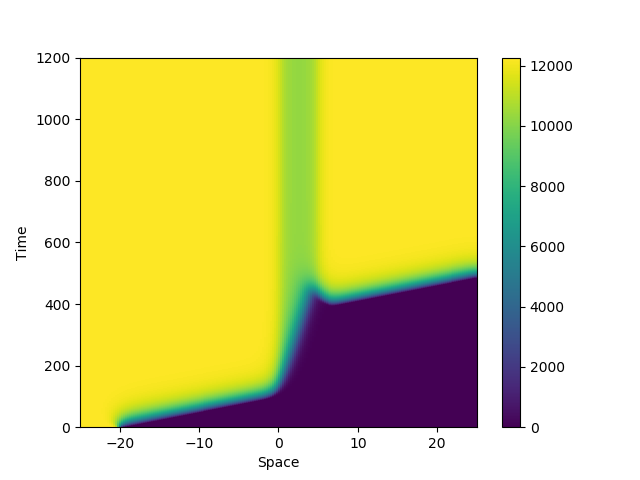}\hfill
  \includegraphics[width=5.5cm]{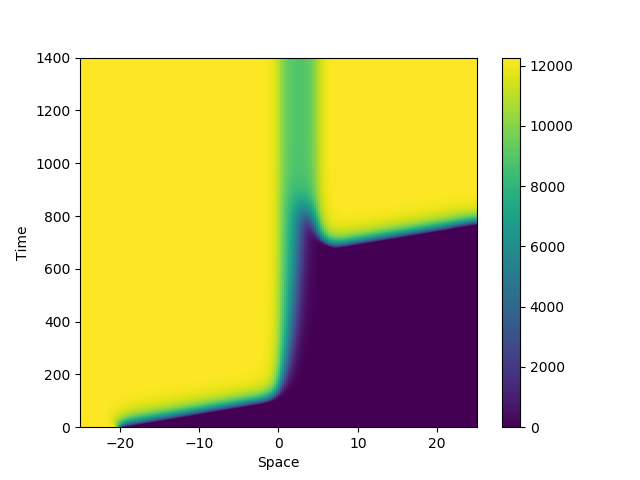}\hfill
  \includegraphics[width=5.5cm]{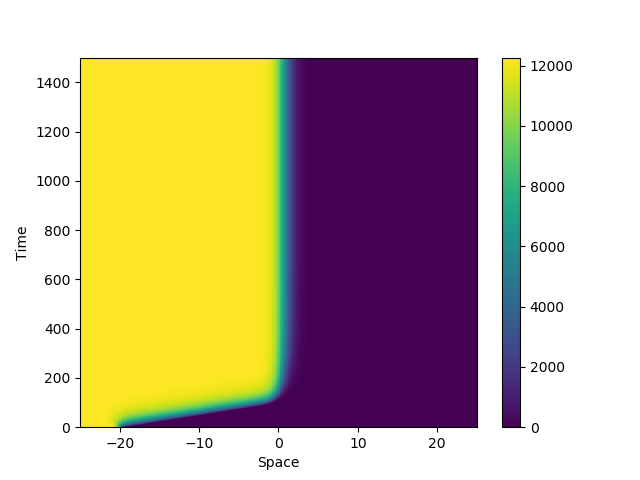}
  \caption{Numerical simulations of the simplified model with $L=5$ km and $\bar{U}=10\,000\text{ km}^{-2}$ (left), $\bar{U}=15\,000\text{ km}^{-2}$ (center), and $\bar{U}=20\,000\text{ km}^{-2}$ (right). To be compared with Fig. \ref{fig1}}
  \label{fig2}
\end{figure}

It is also interesting to observe that when $\beta\to +\infty$, there is no blocking as illustrated in Figure \ref{fig:betainf} for model \eqref{model_x}. This observation may be easily explained for the simplified model. Indeed, when $\beta\to +\infty$, the expression \eqref{eq:g} simplifies into
$$
g(F,0) = \frac{r\nu_E K b F}{b F + K(\nu_E+\mu_E)} - \mu_F F.
$$
Then, $F\mapsto g(F,0)$ admits only two roots $F=0$ and $F=\frac{K(r\nu_E b - \mu_F(\nu_E+\mu_E))}{b \mu_F}$.
Therefore, it is a monostable function, for which the mosquito-free steady state is unstable. As a consequence, if an exponentially small number of mosquitoes cross the region of blocking, it is enough for them to reproduce and to converge to the positive steady state.
We can also verify that when $\beta\to +\infty$ the mosquito-free steady state equilibrium is unstable for model \eqref{model1}. Indeed, putting $M_s=0$ in the system \eqref{model1} and letting $\beta\to +\infty$, the Jacobian matrix of the resulting system at the point $(0,0,0)$ is given by
$$
\begin{pmatrix}
  -\nu_E-\mu_E & 0 & b \\ (1-r)\nu_E & -\mu_M & 0 \\ r \nu_E & 0 & -\mu_F
\end{pmatrix}
$$
Thus, when $b>\mu_F$ this steady state is unstable.

\begin{figure}
  \includegraphics[width=8.5cm]{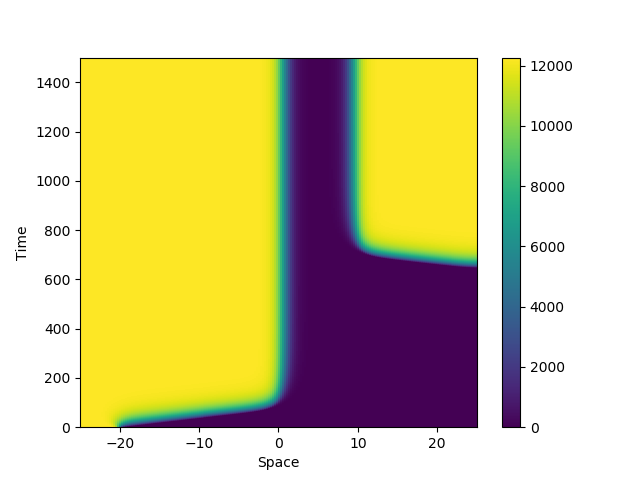}\hfill
  \includegraphics[width=8.5cm]{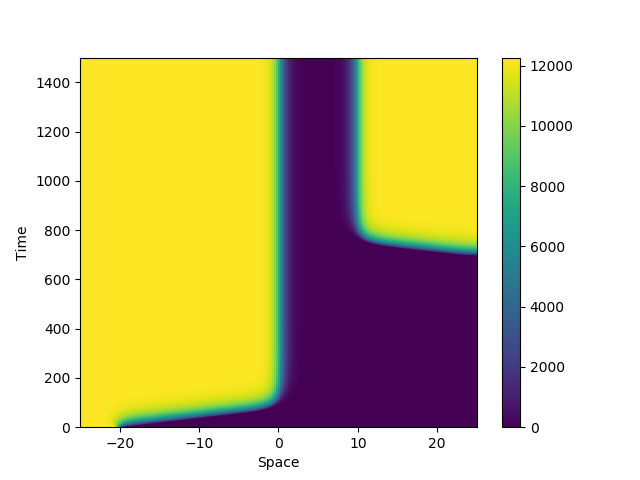}
  \caption{Numerical simulations for the model \eqref{model_x} with $\beta\to +\infty$ and $L=10$\,km~: $\bar{U}=20\,000\text{ km}^{-2}$ (left); $\bar{U}=30\,000\text{ km}^{-2}$ (right).}\label{fig:betainf}
\end{figure}

Finally, we perform some numerical simulations for the system \eqref{model2_x} in order to compare the behaviour of solutions for this system with system \eqref{model_x}.
The time and space dynamics of $F_m$ is displayed in Figures \ref{fig1_model2} and \ref{fig2_model2}.
In Figure \ref{fig1_model2}, the width of the domain $L$ is fixed and we change the intensity of the release $\bar{U}$.
As in Figure \ref{fig1}, we observe that by increasing the intensity of the release $\bar{U}$, we may block the propagation.
In Figure \ref{fig1_model2}, we make the same observation that when $\beta\to +\infty$, the wave is able to cross the active domain, even for $\bar{U}$ and $L$ much larger than what is needed to stop the propagation.
These numerical results are in accordance with what we saw for the previous model.

\begin{figure}[h!]
  \includegraphics[width=5.5cm]{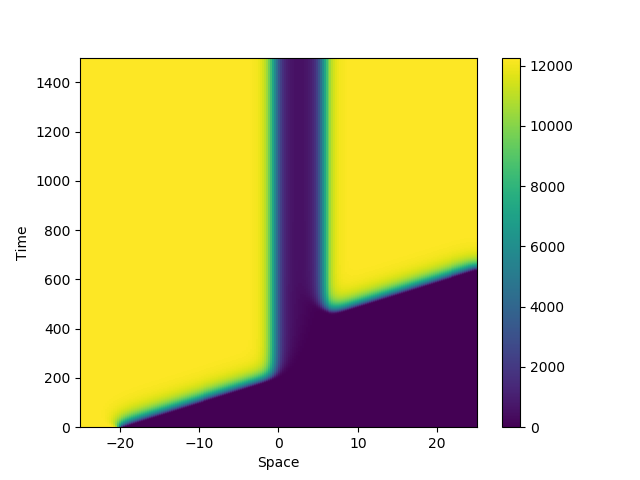}\hfill
  \includegraphics[width=5.5cm]{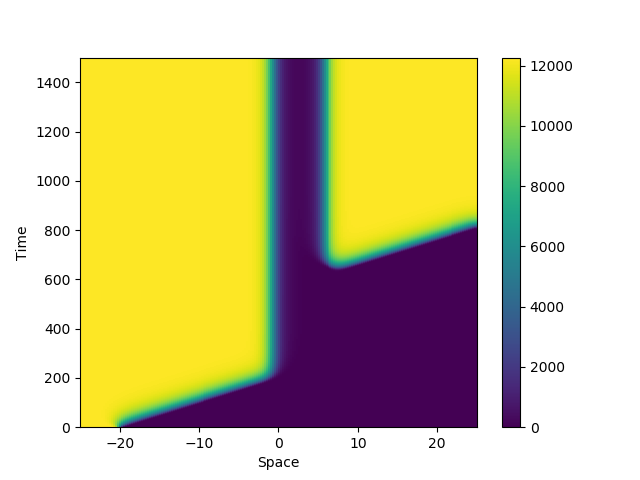}\hfill
  \includegraphics[width=5.5cm]{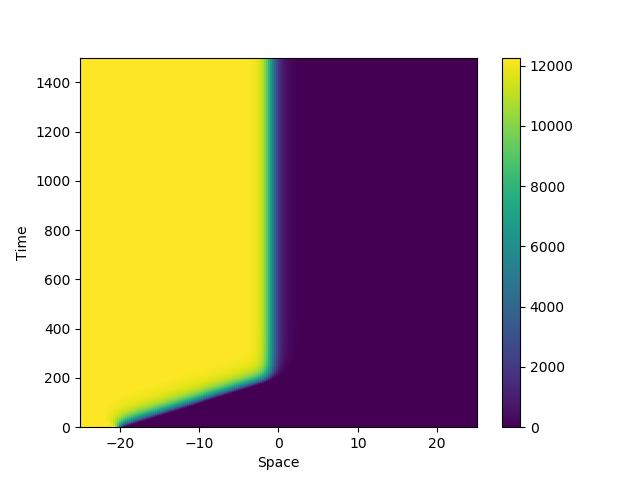}
  \caption{Numerical simulations for model \eqref{model2_x} with $L=5$ km and $\bar{U}=10\,000\text{ km}^{-2}$ (left), $\bar{U}=20\,000\text{ km}^{-2}$ (center), and $\bar{U}=30\,000\text{ km}^{-2}$ (right).}
  \label{fig1_model2}
\end{figure}

\begin{figure}[h!]
  \includegraphics[width=5.5cm]{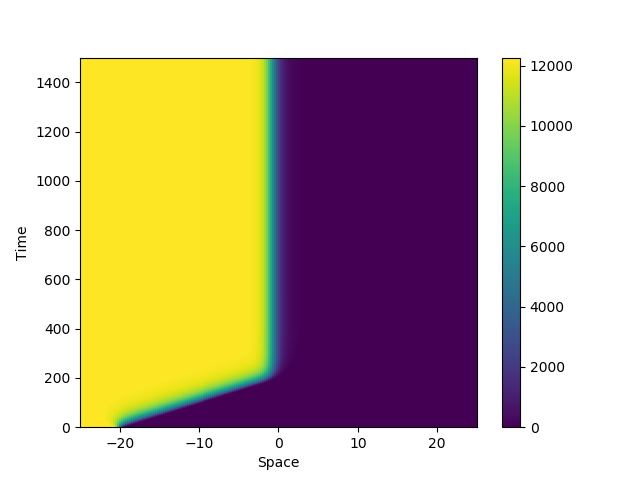}\hfill
  \includegraphics[width=5.5cm]{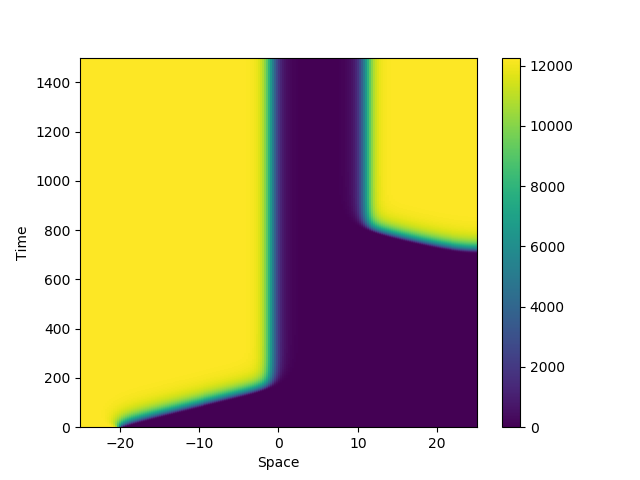}\hfill
  \includegraphics[width=5.5cm]{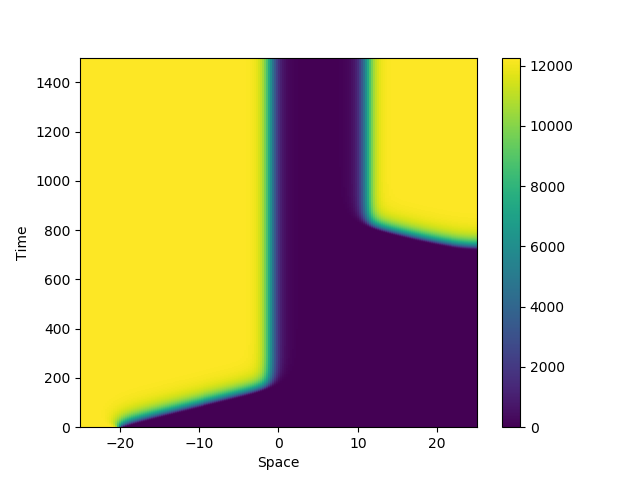}
  \caption{Numerical simulations for model \eqref{model2_x}. Left : $L=10$ km, $\bar{U}=20\,000\text{ km}^{-2}$ and $\beta = 10^{-2}$. Center :  $L=10$ km, $\bar{U}=30\,000\text{ km}^{-2}$ and $\beta\to +\infty$. Right : $L=10$ km, $\bar{U}=40\,000\text{ km}^{-2}$ and $\beta\to +\infty$.}
  \label{fig2_model2}
\end{figure}

\section{Mathematical approach}

These numerical simulations indicate that it should be possible to block the spreading by releasing enough sterile males on a sufficiently wide domain.
However, to be sure that it is not an numerical artefact and that the propagation is really blocked, we have to prove rigorous mathematical results.
The study of wave blocking by local action has been done by several authors with applications for instance in biology or in criminal studies \cite{Pauwelussen, Lewis, Chapuisat, Rodriguez, block, Eberle}.

In \cite{Jorge}, we apply the theory developed e.g. in \cite{Lewis} to prove existence of a blocking for the simplified model.
Let us consider the simpified model \eqref{model_x_simple} with $u=\bar{U} \mathbf{1}_{[0,L]}$.
We call \textbf{barrier} for \eqref{model_x_simple}, any stationnary solution, i.e. any solution $(\widetilde{F},\widetilde{M_s})$ to
\begin{equation}\label{eq:block}
\begin{array}{l}
  \ds - D_u \widetilde{F}'' = g(\widetilde{F},\widetilde{M_s}), \\[2mm]
  \ds - D_u \widetilde{M_s}'' + \mu_s \widetilde{M_s} = \bar{U} \mathbf{1}_{[0,L]}.
\end{array}
\end{equation}

  The main result in \cite{Jorge} is the existence of a barrier for $L$ and $\bar{U}$ large enough:
\begin{theorem}\label{th}
  There exists $L^*>0$ such that
  \begin{itemize}
  \item For $L<L^*$ there is no barrier for \eqref{model_x_simple}.
  \item For $L>L^*$, there exists $\bar{U}^*(L)$ such that for all $ \bar{U} > \bar{U}^*(L)$ there exists a barrier for \eqref{model_x_simple}.  \\
    Moreover, we have $\lim_{L\to L^*} \bar{U}^*(L) = +\infty$, $\bar{U}^*$ is decreasing with respect to $L$, and $\bar{U}^*(L) = O(\frac{1}{(L-L^*)^2})$ as $L\underset{>}{\to} L^*$.
    Furthermore, $\lim_{L\to +\infty} \bar{U}^*(L)$ exists and is bounded from below by $M_\infty$ such that $\int_0^{\bar{F}_2} g(F,M_\infty)\,dF = 0$.
  \end{itemize}
\end{theorem}

\emph{Sketch of the proof:} The proof in \cite{Jorge} is based on the geometric method presented in \cite{Lewis}. The existence of a barrier is linked to the intersection of two associated curves in the phase portrait of $\ds - D_u \widetilde{F}'' = g(\widetilde{F},0)$: the stable manifold of the stable equilibrium $F^+$ and a mapping of the homoclinic orbit that represents the stationary states that tend to zero at infinity. The intersection of these curves allows us to construct a piecewise stationary solution that acts as a barrier for traveling waves potentially arriving from beyond the release zone.

Studying the asymptotic behavior of $g(F,M_s)$ when $M_s\to+\infty$, a lower bound $L*$ is found for the intersection of the curves.  The monotony of the mapping with respect to $L$ and the monotony of $g(F,M_s)$ with respect to $M_s$ then imply that $\bar{U}^*$ is decreasing with respect to $L$. The speed of convergence when $L\to L^*$ is derived from a first order Taylor approximation. 

Finally, using the comparison principle with the parabolic equation $\partial_t \widetilde{F}- D_u \partial_{xx} \widetilde{F} = g(\widetilde{F},\bar{U})$, for which existence of a traveling wave solution with positive velocity when $\bar{U}\leq M_\infty$ is known \cite{AW75}, we deduce that there is no wave-blocking for $\bar{U}\leq M_\infty$, and therefore $M_\infty$ is a lower bound for $\bar{U}^*(L)$ and $\lim_{L\to +\infty} \bar{U}^*(L)$ exists.

%%%%%%%%%%%%%%%%%%%%%%%%%%%%%%%%%%%
%
%%%%%% BIBLIO %%%%%%%%%%%%%%%%%%%%%%
%
%%%%%%%%%%%%%%%%%%%%%%%%%%%%%%%%%%%%

% \bibliography{aggregation.bib}
% \bibliographystyle{plain}
%%%%%%%%%%%%%%%%%%%%%%%%%%%%%%%%%%%%%%%%%%%%%%%%%%%

%%%%%%%%%%%%%%%%%%%%%%%%%%%%%%%%%%%%%%%%%%%%%%%%%%%%%%%%%%%%%%%%%%%%%%%%%%%%%%%%

\end{document}